\documentclass[12pt]{article} 
\usepackage{amsmath}
\usepackage{amsfonts}
\usepackage{amsthm}
\usepackage{amssymb}
\usepackage{epsfig}

\newtheorem{theorem}{Theorem}[section]
\newtheorem{corollary}[theorem]{Corollary}
\newtheorem{lemma}[theorem]{Lemma}

\newtheorem{question}[theorem]{Question}

\theoremstyle{definition}

\theoremstyle{remark}
\newtheorem{remark}{Remark}[section]

\theoremstyle{remark}

\long\def\symbolfootnote[#1]#2{\begingroup\def\thefootnote{\fnsymbol{footnote}}
\footnote[#1]{#2}\endgroup}

\begin{document}

\title{On sumfree subsets of hypercubes}
\author{Daniel J. Katz}

\maketitle

\abstract{We consider the possible sizes of large sumfree sets contained in 
the discrete hypercube $\{1,\ldots,n\}^k$, and we determine upper and lower 
bounds for the maximal size as $n$ becomes large. We also discuss a 
continuous analogue in which our lower bound remains valid and our upper 
bound can be strengthened, and we consider the generalization of both problems
to $l$-fold-sumfree sets.}
\symbolfootnote[0]{2000 Mathematics Subject Classification: 11B75.}

\section{Introduction}

Given an additive group $Z$, we refer to $A \subset Z$ as a sumfree set
if $x+y\neq z$ for all $x,y,z \in A$. (Equivalently, using the notation of
sumsets, $A$ is sumfree if $(A + A) \cap A = \emptyset$.) These sets have been of interest since
at least 1916, when Schur \cite{Schur1916} proved that the positive integers could not be 
partitioned into finitely many such sets.

A common problem in this topic is as follows: given a particular additive group
$Z$ (or perhaps a subset $Z$ of an additive group), how large can a sumfree
subset of $Z$ be, and further, what sort of structure do large sumfree subsets 
have? This problem has been considered for $Z = \mathbb{Z}_{>0}$
\cite{MR1441239, MR1701193},  
$\mathbb{Z}/p \mathbb{Z}$ \cite{MR2254541},
general finite groups (abelian \cite{MR2166359} 
and non-abelian \cite{MR1656927}),
and $\{1,\ldots,n\} \subset \mathbb{Z}$ for arbitrary (usually large) $n$
\cite{MR2134182, MR2142128}. 

The last of these cases suggests a study of 
the ``discrete hypercube'' $Z = \{1,\ldots,n\}^k \subset \mathbb{Z}^k$ for $k > 1$.
In particular, we would like to know how proportionately large a sumfree subset of
$\{1,\ldots,n\}^k$ can be when $n$ is large. For this purpose, we define
$$c_k := \limsup_{n \to \infty} \frac{1}{n^k} \max \{\#S:S \in \{1,\ldots,n\}^k \mbox{ is
sumfree}\}.$$

Previous work on sumfree subsets of $\{1,\ldots,n\}$ has shown that
$c_1 = 1/2$. (The set of odd elements, for example, is optimally large
for all $n$.) Let $S$ be a sumfree subset of $\{1,\ldots,n\}^k$ of size $\alpha n^k$, 
and let $k' > k$. The inverse image $S'$ of a natural projection from 
$\{1,\ldots,n\}^{k'}$ to $\{1,\ldots,n\}^k$ is also sumfree, and has size 
$\alpha n^{k'}$. Using this fact, it is clear that $c_{k'} \geq c_k$ for $k' > k$,
and thus $1/2 \leq c_k \leq 1$ for all $k$.

The largest sumfree subsets we have observed in the square
$\{1,\ldots,n\}^2$ take the form of thick diagonal ``stripes''; generalizing this
construction, we can construct large sumfree subsets in $\{1,...,n\}^k$ and
thus prove a general lower bound for $c_k$. 

\begin{theorem}
Defining $c_k$ as above,
$$c_k \geq 1 - \frac{2}{k!}\sum^{\lfloor k/3 \rfloor}_{i=0} (-1)^i \binom{k}{i} \left(\frac{k}{3}-i\right)^k.$$
\end{theorem}

Analysis of this lower bound yields:

\begin{corollary}
Defining $c_k$ as above,
$$\lim_{k \to \infty} c_k = 1.$$
\end{corollary}

We also prove a general upper bound for $c_k$ using a
combinatorial method, although it is difficult to write this bound as an 
explicit function of $k$. 

\begin{theorem}
Defining $c_k$ as above, let $\alpha^*$ be the unique root in $[1/2,1]$ of the equation
$$\alpha = (2-2\alpha)\left(1+ \sum^k_{i=0} \frac{1}{i!}\left(\ln \frac{1}{2-2\alpha}\right)^i\right).$$
Then $$c_k \leq \alpha^*.$$
\end{theorem}

Our approach to the upper bound depends on the idea that if an element 
of a sumfree subset $S \in \{1,\ldots,n\}^k$ is the sum of many pairs of elements, 
none of these pairs can be in $S$. This means that if  $S$ contains a certain 
proportion of the full set, a certain number of elements cannot belong to $S$, which 
causes a contradiction if the proportion is large. 

%

To give an idea of the distance between our lower and upper bounds, here are the approximate
bounds given by these theorems for $2 \leq k \leq 6$:

\begin{eqnarray*}
0.555556 \leq & c_2 & \leq 0.913875 \\
0.666667 \leq & c_3 & \leq 0.942361 \\
0.740741 \leq & c_4 & \leq 0.961192 \\
0.796639 \leq & c_5 & \leq 0.973763 \\
0.838889 \leq & c_6 & \leq 0.982208 
\end{eqnarray*}

Our calculations for both the lower and upper bounds involve approximating 
numbers of lattice points in $\{1,\ldots,n\}^k$ by integrating over subsets of
$[0,n]^k$. This approximation is less than exact, but the error
becomes trivial compared to $n^k$ when $n$ is large, and thus it
ultimately does not affect the value of $c_k$. These integrals become more 
complicated as $k$ grows, but they can be calculated explicitly by induction, 
where counting the lattice points directly becomes cumbersome in higher 
dimensions.

This integral method actually suggests a non-discrete version of the problem:
maximizing the volume of Lebesgue-measurable sumfree subsets of the 
``continuous hypercube'' $[0,1]^k \subset \mathbb{R}^k$. We will see that the bounds 
we calculated in Theorems 1.1 and 1.3 hold in this setting, and in fact the upper 
bound can be improved by applying an iteration process.

We will also discuss some results that generalize our processes to
$l$-fold-sumfree sets; that is, sets $A$ such that $x_1 + \cdots + x_l \neq z$ for 
all $x_1, \ldots, x_l, z \in A$. The lower bound for sumfree sets extends easily to
$l > 2$; the upper bound is difficult to apply when $l > 4$, but interestingly in the
$l=3$ case it gives a bound which is explicit rather than the root of an equation.

Finally, we will present some concluding remarks, suggesting two divergent paths
for future investigation in the subject.

\section{Introductory lemmas}

In order to bound the constants under consideration, we will need the following volume formula.

\begin{lemma}
Given $a \in [0,k]$, the volume of the region
$$\{(x_1,\ldots,x_k) \in [0,1]^k:x_1 + \cdots + x_k \leq a\} \subset \mathbb{R}^k$$
is equal to
$$\frac{1}{k!}\sum^{\lfloor a \rfloor}_{i=0} (-1)^i \binom{k}{i} (a-i)^k.$$
\end{lemma}

\begin{proof}
This is a special case of Theorem 1 in Section I.9 of \cite{MR0210154}. 
\end{proof}

\begin{remark}
The proof in \cite{MR0210154} uses probability theory, but the formula can also be obtained directly using an inclusion-exclusion argument. The latter proof is useful in that it can easily be adapted to count the number of lattice points in the region; however, we will not use this formula, so we omit the alternate proof.
\end{remark}

%
%


We will also need the following integral formula, easily proven by induction.

\begin{lemma}
$$ 
\int_c^1 dx_1
\int_{c/x_1}^1 dx_2
\int_{c/x_1 x_2}^1 dx_3
\cdots
\int_{c/x_1 \cdots x_{k-1}}^1 dx_k
= 1 - c  \sum_{i=0}^{k-1} \frac{1}{i!} \left(\ln \frac{1}{c}\right)^i
$$
\end{lemma}

\begin{remark}
In a sense, the domain of integration is a multiplicative analogue of the $k$-simplex. Also note that the right side of the equation approaches 0 as $k \to \infty$, since the sum is a truncated Maclaurin series for $e^x$ evaluated at $x = - \ln c$ .
\end{remark}

\begin{proof}
Let $J(k,c)$ represent the left side of the equation. The theorem is 
clearly true when $k=1$, so we proceed by induction on $k$. Assume the statement is true
for $k$; then

\begin{eqnarray*}
J(k+1,c) & = & \int_c^1 J\left(k,\frac{c}{x_1}\right) dx_1 \\
& = & \int_c^1 \left(1 - \frac{c}{x_1} \sum_{i=0}^{k-1} \frac{1}{i!} (\ln x_1 - \ln c)^i \right) dx_1 \\
& = & (1-c) - c \sum_{i=0}^{k-1} \frac{1}{i!} \int_0^{-\ln c} (\ln x_1 - \ln c)^i d(\ln x_1 - \ln c) \\
& = & (1-c) - c \sum_{i=0}^{k-1} \frac{1}{(i+1)!} (- \ln c)^{i+1} \\
& = & 1 - c  \sum_{i=0}^{k-1} \frac{1}{i!} \left(\ln \frac{1}{c}\right)^i
\end{eqnarray*}

and thus the lemma holds for all $k$.
\end{proof}

Finally, we quote a theorem from Lang, adapted for our purposes, which will allow us to use integrals to approximate subsets of lattices.

\begin{theorem}[Lang]
Let $D$ be a subset of $[0,1]^k$ such that the boundary of $D$ has a Lipschitz-continuous parametrization in $(k-1)$ variables, and let
$nD = \{nx : x \in D\}$. Then
$$\#\left(\{1,\ldots,n\}^k \cap nD\right) = n^k \mathrm{Vol}(D) + O(n^{k-1}).$$
\end{theorem}

\begin{proof}
Apply Theorem 2, p. 128 in \cite{MR1282723} with $L = \mathbb{Z}^k$ and $F=(0,1]^k$. There are fewer than $k(n+1)^{k-1}$ lattice points in the intersection of $nD$ and $(\{0,\ldots,n\}^k \backslash \{1,\ldots,n\}^k)$, and these can be absorbed into the error term.
\end{proof}

\section{Bounding $c_k$ from below}

One method of generating sumfree sets in $\{1,\ldots,n\}$ is to consider ``cross-section'' sets
$$K_a := \{(x_1,\ldots,x_k) \in \{1,\ldots,n\}^k:x_1 + \cdots + x_k = a\}.$$

If $A$ is a sumfree set in $\{k,\ldots,kn\}$, the set $S = \cup_{a \in A} K_a$ is sumfree, because
if $(x_1,\ldots,x_k)$ and $(y_1,\ldots,y_k)$ are both contained in $S$, then
$$(x_1+y_1) + \cdots (x_k + y_k) = (x_1 + \cdots + x_k) + (y_1 + \cdots + y_k) \notin A,$$
so the sum of these two elements is not in $S$.

We will determine a lower bound for $c_k$ using sets of the form
$$S(n,k,a) := \{(x_1,\ldots,x_k) \in \{1,\ldots,n\}^k:a \leq x_1 + \cdots + x_k < 2a\}.$$
Since $\{a+1,\ldots,2a\}$ is clearly sumfree in $\{k,\ldots,kn\}$, $S(n,k,a)$ is sumfree.
To obtain an optimal lower bound for this method, we need to choose a value of
$a$ that maximizes the size of $S(n,k,a)$. We approximate this size
using the region
$$\widetilde{S}(n,k,a) := \{(x_1,\ldots,x_k) \in [0,n]^k \subset \mathbb{R}^k: a \leq x_1 + \cdots + x_k < 2a\}.$$

Note that since $\widetilde{S}(1,k,a)$ is just a scaled-down copy of $\widetilde{S}(n,k,an)$, we have
$$\widetilde{S}(n,k,an) = n^k \widetilde{S}(1,k,a).$$

\begin{proof}[Proof of Theorem 1.1]
By Lemma 2.1, the volume of 
$$\{(x_1,\ldots,x_k) \in [0,1]^k \subset \mathbb{R}^k: x_1 + \cdots + x_k < a\}$$
is equal to 
$$V_1(k,a) := \frac{1}{k!}\sum^{\lfloor a \rfloor}_{i=0} (-1)^i \binom{k}{i} (a-i)^k.$$
Changing variables, the volume of
$$\{(x_1,\ldots,x_k) \in [0,1]^k \subset \mathbb{R}^k :x_1 + \cdots + x_k > 2a\}$$ 
is equal to 
$$V_2(k,a) := \frac{1}{k!}\sum^{\lfloor k-2a \rfloor}_{i=0} (-1)^i \binom{k}{i} (k-2a-i)^k.$$

We wish to choose a value of $a$ (for each $k$) which maximizes 
$$\mathrm{Vol}(\widetilde{S}(1,k,a)) = 1 - V_1(k,a) - V_2(k,a).$$

A computer search (for $k < 60$) suggests the optimal choice satisfies
$a = k/3 + O(1)$, although it is difficult to determine an exact formula. For our lower 
bound, we choose $a=k/3$; this value appears to be close to optimal, and it gives a 
concise expression for $\mathrm{Vol}(\widetilde{S}(1,k,a))$ (since $V_1(k,k/3) = V_2(k,k/3)$). 
The regions $\widetilde{S}(1,k,k/3)$ for $k=2,3$ are shown in Figures 1 and 2.

\begin{figure}[h]
\begin{minipage}[b]{0.5\linewidth} 
\centering
\epsfig{file=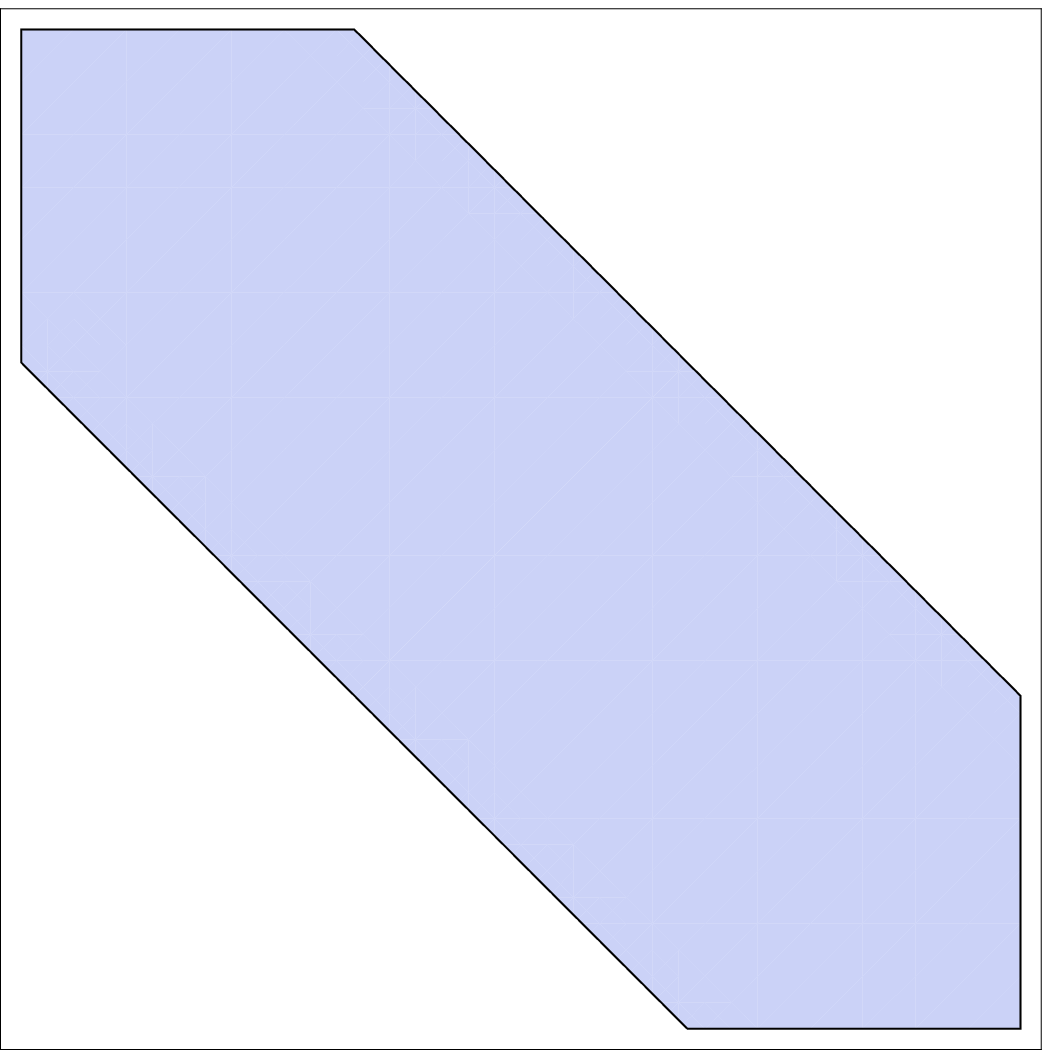, scale=0.5}
\caption{Sumfree region for k = 2}
\end{minipage}
\hspace{0.5cm} 
\begin{minipage}[b]{0.5\linewidth}
\centering
\epsfig{file=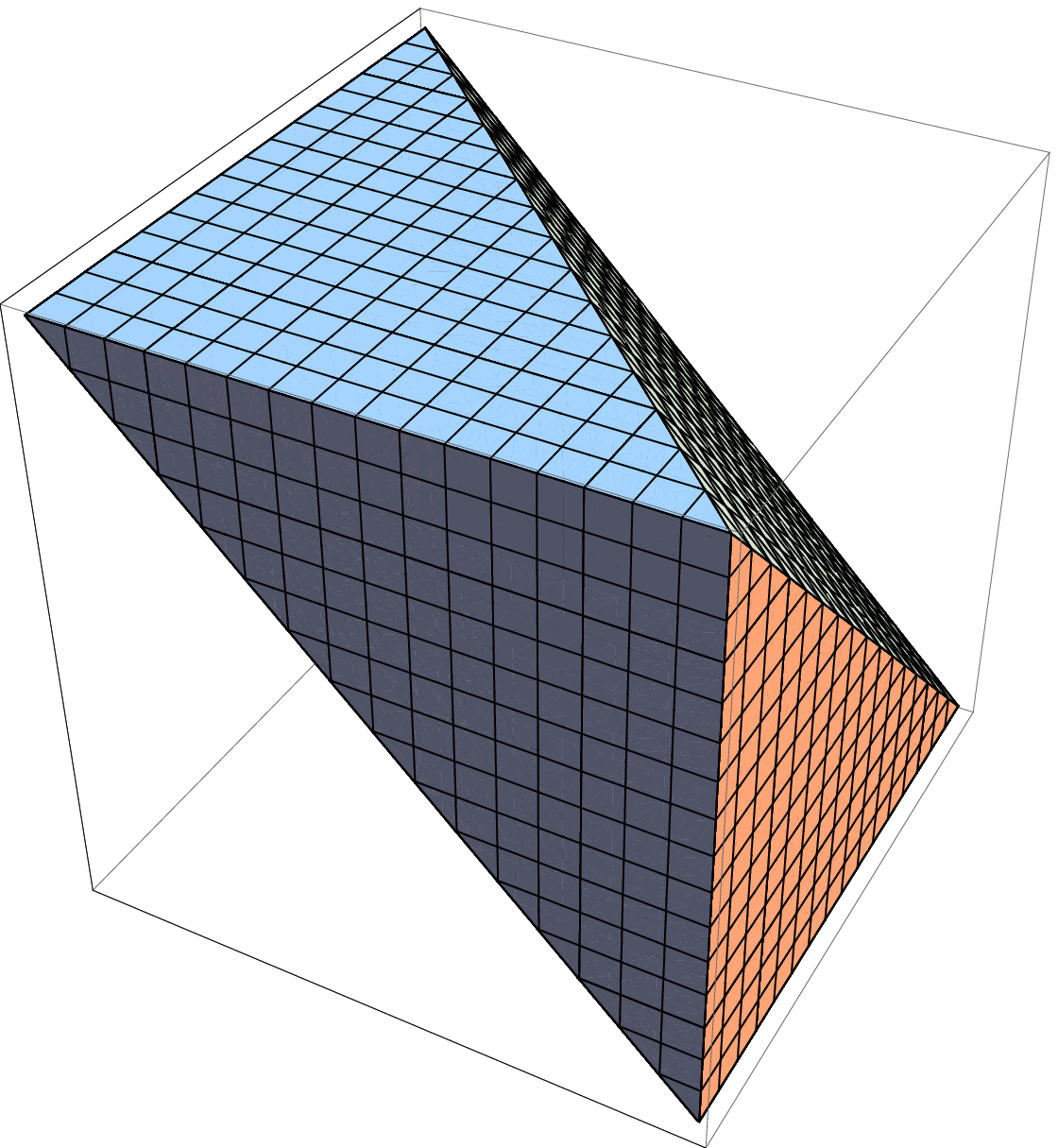, scale=0.5}
\caption{Sumfree region for k = 3}
\end{minipage}
\end{figure}

By Theorem 2.3, $\#S(n,k,kn/3) = \mathrm{Vol}(\widetilde{S}(n,k,kn/3)) + O(n^{k-1})$, since
all of the boundaries of the region are hyperplanes and are thus Lipschitz parametrizable. Then we have
\begin{eqnarray*}
c_k & \geq &  \limsup_{n \to \infty} \frac{1}{n^k} \#(S(n,k,kn/3)) \\
& = & \limsup_{n \to \infty} \frac{1}{n^k} \left(\mathrm{Vol}(\widetilde{S}(n,k,kn/3)) + O(n^{k-1})\right) \\
& = & \limsup_{n \to \infty} \frac{1}{n^k} \left(n^k \mathrm{Vol}(\widetilde{S}(1,k,k/3)) + O(n^{k-1})\right) \\
& = & \mathrm{Vol}(\widetilde{S}(1,k,k/3)) \\
& = & 1 - \frac{2}{k!}\sum^{\lfloor k/3 \rfloor}_{i=0} (-1)^i \binom{k}{i} \left(\frac{k}{3}-i\right)^k.
\end{eqnarray*}\end{proof}

To determine the behavior of this lower bound, we need the following lemma.

\begin{lemma}
Let $a,b$ satisfy $0 < a < b<\frac{1}{2}$ and 
$$\frac{1}{3} - a < \frac{b^b (1-b)^{(1-b)}}{e}.$$
Then
$$\lim_{k \to \infty} \frac{1}{k!} \sum^{\lfloor bk \rfloor}_{i=\lceil ak \rceil} \binom{k}{i} \left(\frac{k}{3}-i\right)^k = 0.$$
\end{lemma}

\begin{proof}
\begin{eqnarray*}
\lim_{k \to \infty} \frac{1}{k!} \sum^{\lfloor bk \rfloor}_{i=\lceil ak \rceil} \binom{k}{i} \left(\frac{k}{3}-i\right)^k
& \leq & \lim_{k \to \infty} \frac{1}{k!} \sum^{\lfloor bk \rfloor}_{i=\lceil ak \rceil} \binom{k}{bk} \left(\frac{k}{3}-ak\right)^k \\
& \leq & \lim_{k \to \infty} \frac{(b-a+1)k}{(bk)!(k-bk)!}k^k \left(\frac{1}{3}-a \right)^k.
\end{eqnarray*}

Then, using Stirling's approximation,
\begin{eqnarray*}
\lim_{k \to \infty} \frac{(b-a+1)k}{(bk)!(k-bk)!}k^k \left(\frac{1}{3}-a \right)^k
& = & \lim_{k \to \infty} \frac{b-a+1}{2\pi \sqrt{b(1-b)}} \left(
\frac{\left(\frac{1}{3}-a\right)e}{b^b (1-b)^{(1-b)}} \right)^k \\
& = & 0. \end{eqnarray*}
\end{proof}

\begin{proof}[Proof of Corollary 1.2]
Define a sequence $\{a_i\}$ as follows:
$$a_0 = \frac{1}{3},$$
$$a_{i+1} = \frac{1}{3}-\frac{a_i^{a_i} (1-a_i)^{(1-a_i)}}{e}.$$

Calculating the initial terms of this sequence, we find that $a_7 < 0$, and that $a_1,
a_2,\ldots,a_6$ are irrational. Thus we can split the sum as follows:
$$\sum^{\lfloor k/3 \rfloor}_{i=0} (-1)^i \binom{k}{i} \left(\frac{k}{3}-i\right)^k = \sum^{\lfloor a_0 k \rfloor}_{i=\lceil a_1 k \rceil}
+ \sum^{\lfloor a_1 k \rfloor}_{i=\lceil a_2 k \rceil}
+ \cdots + \sum^{\lfloor a_6 k \rfloor}_{i= 0}.$$

By Lemma 3.1, each of these partial sums approaches zero as $k$ approaches infinity,
and so the entire sum does as well. This means that the lower bound determined in Theorem 1.1
approaches 1 as $k$ grows, and therefore so does $c_k$.
\end{proof}

\section{Bounding $c_k$ from above}

The process of finding an upper bound for $c_k$ is a bit more complicated, since we
cannot do so simply by exhibiting a sumfree set. Here our procedure is to assume 
our sumfree set has a certain size, and from this we determine a contradiction if the
set is too large. 

\begin{proof}[Proof of Theorem 1.3]
Let $S$ be a sumfree subset of $\{1,\ldots,n\}^k$ with $\#S \geq \alpha n^k$.  
Suppose $b = (b_1, \ldots, b_k)$ is an element of $S$ with component values
close to 1. 
There are $\frac{1}{2} \prod_{i=1}^k (b_i-1)$ disjoint pairs of elements in $\{1,\ldots,n\}^k$ 
which sum to $b$, unless all of the $b_i$'s are even, in which case there are
$\frac{1}{2} (\prod_{i=1}^k (b_i-1) + 1)$, to account for the point
$(\frac{b_1}{2}, \ldots, \frac{b_k}{2})$... Either way, the number of pairs is equal to 
$\frac{1}{2} b_1 \cdots b_n + O(n^{k-1})$. 

At least one element from each of these pairs must be absent from
$S$, so (approximately)
$$\alpha n^k \leq \#S \leq n^k - \frac{1}{2}b_1b_2 \cdots b_k + O(n^{k-1})$$
and thus
$$ b_1b_2 \cdots b_k \leq (2 - 2\alpha) n^k + O(n^{k-1}) = \beta n^k,$$
where $\beta = (2-2\alpha) + O(1/n)$.

This ``disqualifies'' a number of lattice points from being contained in $S$, namely
$$T(n,k, \alpha) = \{(b_1,\ldots, b_k) \in \{1,\ldots,n\}^k : b_1b_2 \cdots b_k > \beta n^k\}.$$

As in the last section, we will approximate this collection of lattice points by the region

$$\widetilde{T}(n,k, \alpha) = \{(b_1,\ldots, b_k) \in [0,n]^k \subset \mathbb{R}^k : b_1b_2 \cdots b_k > \beta n^k\}.$$

We can calculate the volume of $\widetilde{T}(n,k, \alpha)$ using an integral:

\begin{eqnarray*}
\mathrm{Vol}(\widetilde{T}(n,k, \alpha)) & = &
\int_{\beta n}^n dx_1
\int_{\beta n/x_1}^n dx_2
\int_{\beta n/x_1 x_2}^n dx_3
\cdots
\int_{\beta n/x_1 \cdots x_{k-1}}^n dx_k \\
& = & n^k 
\int_{\beta }^ dx_1
\int_{\beta /x_1}^ dx_2
\int_{\beta /x_1 x_2}^ dx_3
\cdots
\int_{\beta /x_1 \cdots x_{k-1}}^ dx_k \\
& = & n^k 
\left(1 - \beta  \sum_{i=0}^{k-1} \frac{1}{i!}\left(\ln \frac{1}{\beta}\right)^i \right),
\end{eqnarray*}
using Lemma 2.2 with $c = \beta$ in the final step. Since, by Theorem 2.3, $\mathrm{Vol}(T(n,k, \alpha)) = \mathrm{Vol}(\widetilde{T}(n,k, \alpha)) + O(n^{k-1})$, this indicates that for any $\alpha$
such that

$$
\alpha = 1 - \frac{1}{2} \beta + O(1/n) \geq f(\beta) := \beta \sum_{i=0}^{k-1} \frac{1}{i!} \left(\ln \frac{1}{\beta}\right)^i,
$$
any set larger than $\alpha n^k$ is simultaneously smaller than $\alpha n^k$, yielding a
contradiction. 

Observe that 
\begin{eqnarray*}
f'(\beta) & = & \beta \left( \sum_{i=0}^{k-2} \frac{1}{i!}\left(\ln \frac{1}{\beta}\right)^i \right) \left(-\frac{1}{\beta}\right)
+ \sum_{i=0}^{k-1} \frac{1}{i!}\left(\ln \frac{1}{\beta}\right)^i \\
& = & \frac{1}{(k-1)!}\left(\ln \frac{1}{\beta}\right)^{k-1} > 0.
\end{eqnarray*}

Thus, as $\beta$ increases from $0$ to $1$, $f(\beta)$ increases monotonically from $0$ to $1$,
while $(1 - \beta/2)$ decreases monotonically from $1$ to $1/2$. 
Therefore, the equation $1 - \beta/2 = f(\beta)$ has a unique root $\beta^* \in [0,1]$, and letting
$\alpha^* = (1-\beta^*/2)$, we must 
have $\alpha < \alpha^*+O(1/n)$ to avoid a contradiction.
Letting $n$ approach infinity, we conclude that $c_k \leq \alpha^*$.
\end{proof}

\section{A continuous analogue}

In the previous two sections, we used the volume of continuous regions to estimate
the size of discrete sets. Alternatively, we could have asked our question about the
continuous regions in the first place. Let us consider
$$\tilde{c}_k :=  \max \{\mbox{Vol}(S):S \in [0,1]^k \mbox{ is measurable and sumfree}\}.$$

\begin{theorem}
Defining $\tilde{c}_k$ as above,
$$\tilde{c}_k \geq 1 - \frac{2}{k!}\sum^{\lfloor k/3 \rfloor}_{i=0} (-1)^i \binom{k}{i} \left(\frac{k}{3}-i\right)^k.$$
\end{theorem}

\begin{corollary}
Defining $\tilde{c}_k$ as above,
$$\lim_{k \to \infty} \tilde{c}_k = 1.$$
\end{corollary}

\begin{theorem}
Defining $\tilde{c}_k$ as above, let $\alpha^*$ be the unique root in $[1/2,1]$ of the equation
$$\alpha = (2-2\alpha)\left(1+ \sum^k_{i=0} \frac{1}{i!}\left(\ln \frac{1}{2-2\alpha}\right)^i\right).$$
Then $$\tilde{c}_k \leq \alpha^*.$$
\end{theorem}

\begin{proof}
The proofs of these statements are virtually identical to the proofs of Theorem 1.1, Corollary 1.2,
and Theorem 1.3 respectively. The only difference is that $S(n,k,a)=\widetilde{S}(n,k,a)$ and 
$T(n,k,a)=\widetilde{T}(n,k,a)$, so there are no error terms to incorporate.

The proof of Theorem 5.3 warrants one additional comment. If $S$ is a Lebesgue-measurable sumfree set,
and $(b_1,\ldots, b_k) \in S$, then the sets
$$S_{b_1,\ldots,b_k} := S \cap ([0,b_1] \times \cdots \times [0,b_k])$$ 
$$S'_{b_1,\ldots,b_k} := \{(b_1,\ldots, b_k)-x:x \in S_{b_1,\ldots,b_k}\}$$
are disjoint sets of equal volume contained in $[0,b_1] \times \cdots \times [0,b_k]$.
Therefore we have
$$\mathrm{Vol}(S_{b_1,\ldots,b_k}) \leq \frac{1}{2}b_1 \cdots b_k.$$
This substitutes for the combinatorial argument that begins the proof of 
Theorem 1.3. 
\end{proof}

The upper bound for $\tilde{c}_k$ may be improved by a slightly different approach. 
Recall the definition of $f$ from the proof of Theorem 1.3, and suppose 
$\mbox{Vol}(S) = \alpha$, where $\alpha = f(\alpha)$. 
This would require $S$ to consist of all of $[0,1]^k$ except the ``integral wedge'' $\widetilde{T}(n,k,a)$ that we 
removed from the upper right corner. But this would mean that $S$
contains \emph{all} of a smaller set $[0,m]^k$. Scaling by a factor of $1/k$, this violates the upper bound 
we've just determined. We can improve our upper bound by exploiting this condition and iterating the 
process.

\begin{theorem}
Defining $\tilde{c}_k$ as above,
let $\alpha^{**}$ be the unique root in $(1/2,1)$ of the equation
$$\alpha = \frac{1}{2} - \alpha + \sum^{\infty}_{i=k} \frac{1}{i!}\left(\ln \frac{1}{2-2\alpha}\right)^i.$$
Then $$\tilde{c}_k \leq \alpha^{**}.$$
\end{theorem}

\begin{proof}
If $\mbox{Vol}(S) = \alpha$, consider the set $S' = [0,(2-2\alpha)^{1/k}]^k \in [0,n]^k$, which is
disjoint from $\widetilde{T}(n,k, \alpha)$ except for a single point. Since $S$ cannot
intersect $\widetilde{T}(n,k, \alpha)$, the \emph{smallest} density $(S \cap S')/S'$ we can 
achieve is

\begin{eqnarray*}
\varphi_k(\alpha) & := & \frac{\alpha - (1 - \mathrm{Vol}(\widetilde{T}(n,k, \alpha)) - (2-2\alpha))}{(2-2\alpha)} \\
& = & \frac{\mathrm{Vol}(\widetilde{T}(n,k, \alpha))}{(2-2\alpha)} + \frac{1}{2} \\
& = & \frac{\alpha}{2-2\alpha} - \sum^{k-1}_{i=1} \frac{1}{i!}\left(\ln \frac{1}{2-2\alpha}\right)^i,
\end{eqnarray*}
where again we apply Lemma 2.2 in the final step.

If any $\varphi_k^m(\alpha) > 1$ (that is, the $m$th iteration of $\varphi_k$, not the $m$th
power), we have a contradiction. We wish to show that the function
$$\psi_k(\alpha) = \varphi_k(\alpha) - \alpha$$
has a unique root $\alpha^{**}$ on the interval $(0.5,1)$, and that any $\alpha > \alpha^{**}$
will grow larger than 1 through repeated application of $\psi_k$. First we observe that $\psi_k(0.5) = 0$.
On the interval $(1/2,1)$,
\begin{eqnarray*}
\psi_k(\alpha) & = & \frac{\alpha}{2-2\alpha}-\left(\frac{1}{2-2\alpha} - 1 - \sum^{\infty}_{i=k} 
\frac{1}{i!}\left(\ln \frac{1}{2-2\alpha}\right)^i \right) - \alpha \\
& = & \frac{1}{2} - \alpha + \sum^{\infty}_{i=k} \frac{1}{i!}\left(\ln \frac{1}{2-2\alpha}\right)^i.
\end{eqnarray*}

Next we determine the first and second derivatives:
\begin{eqnarray*}
\psi_k'(\alpha) & = &
\frac{1}{1-\alpha} \sum^{\infty}_{i=k-1} \frac{1}{i!}\left(\ln \frac{1}{2-2\alpha}\right)^i - 1. \\
\psi_k''(\alpha) & = &
\frac{1}{1-\alpha} \left(\sum^{\infty}_{i=k-2} \frac{1}{i!}\left(\ln \frac{1}{2-2\alpha}\right)^i \right) \\*
& & \hspace{15 mm}
+\frac{1}{(1-\alpha)^2} \left(\sum^{\infty}_{i=k-1} \frac{1}{i!}\left(\ln \frac{1}{2-2\alpha}\right)^i \right).
\end{eqnarray*}

Inspecting these derivatives, we see that $\psi_k'(1/2) = -1 < 0$, and $\psi_k''$ is positive on the interval
$(1/2,1)$. Thus, $\psi_k$ has at most one root on the interval. Finally, since the quantity $(\ln \frac{1}{2-2\alpha})$ approaches infinity as $\alpha$ approaches 1 from below, it is clear that 
$$\lim_{\alpha \to 1^-} \psi_k'(\alpha) = -\infty.$$

This implies that $\psi_k(\alpha) = \varphi_k(\alpha) - \alpha$ has a root $\alpha^{**} \in (1/2,1)$, and 
furthermore, since $\psi_k$ is increasing for $\alpha > \alpha^{**}$, iteration of $\varphi_k$ on any $\alpha > \alpha^{**}$
will eventually give a result larger than 1. Thus we must have $\tilde{c}_k \leq \alpha^{**}$.
\end{proof}

Figure 3 illustrates the method applied to prove Theorem 5.3, in which one region of the hypercube is ruled out, while Figure 4 illustrates the method of Theorem 5.4 , in which successive regions are removed from hypercubes of decreasing size.

\begin{figure}[h]
\begin{minipage}[b]{0.5\linewidth} 
\centering
\epsfig{file=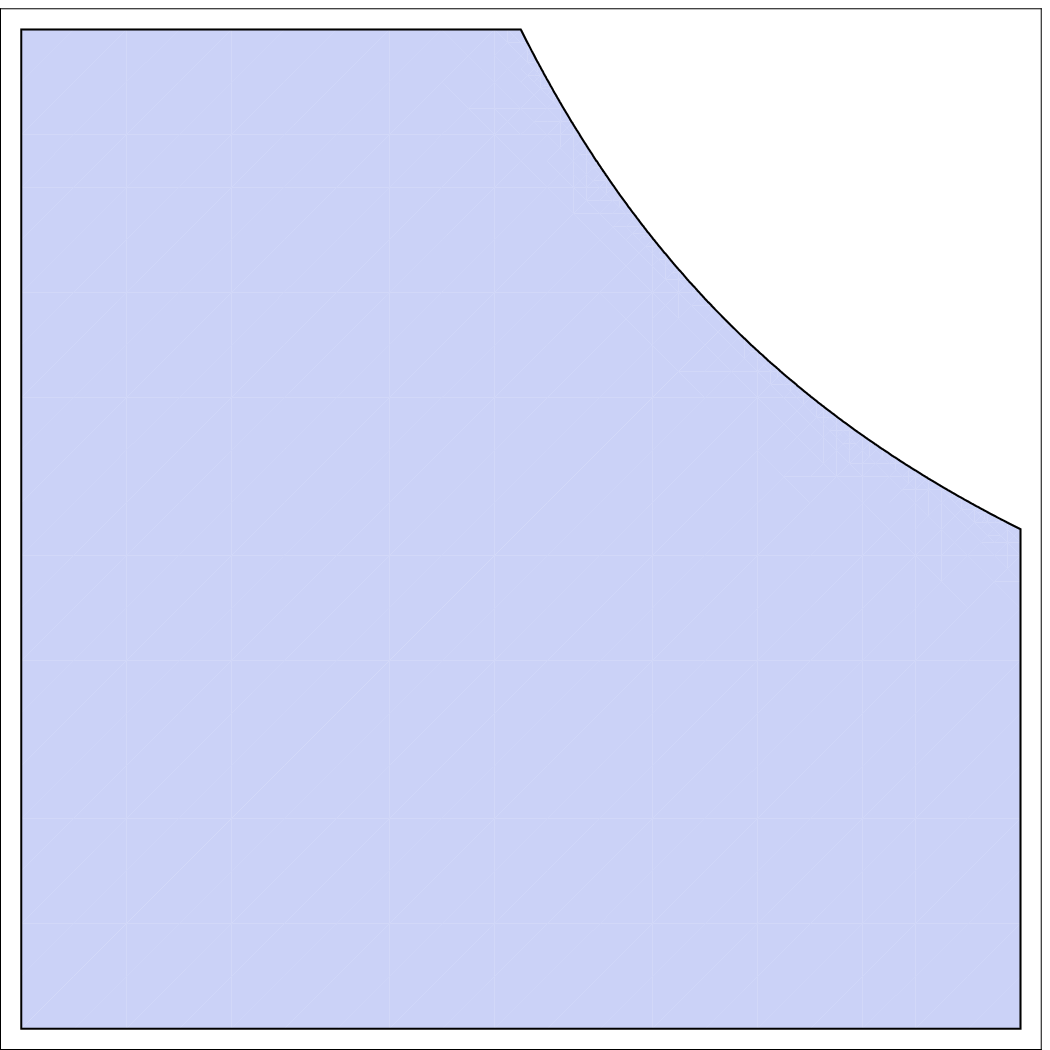, scale=0.5}
\caption{Method of Theorem 5.3}
\end{minipage}
\hspace{0.5cm} 
\begin{minipage}[b]{0.5\linewidth}
\centering
\epsfig{file=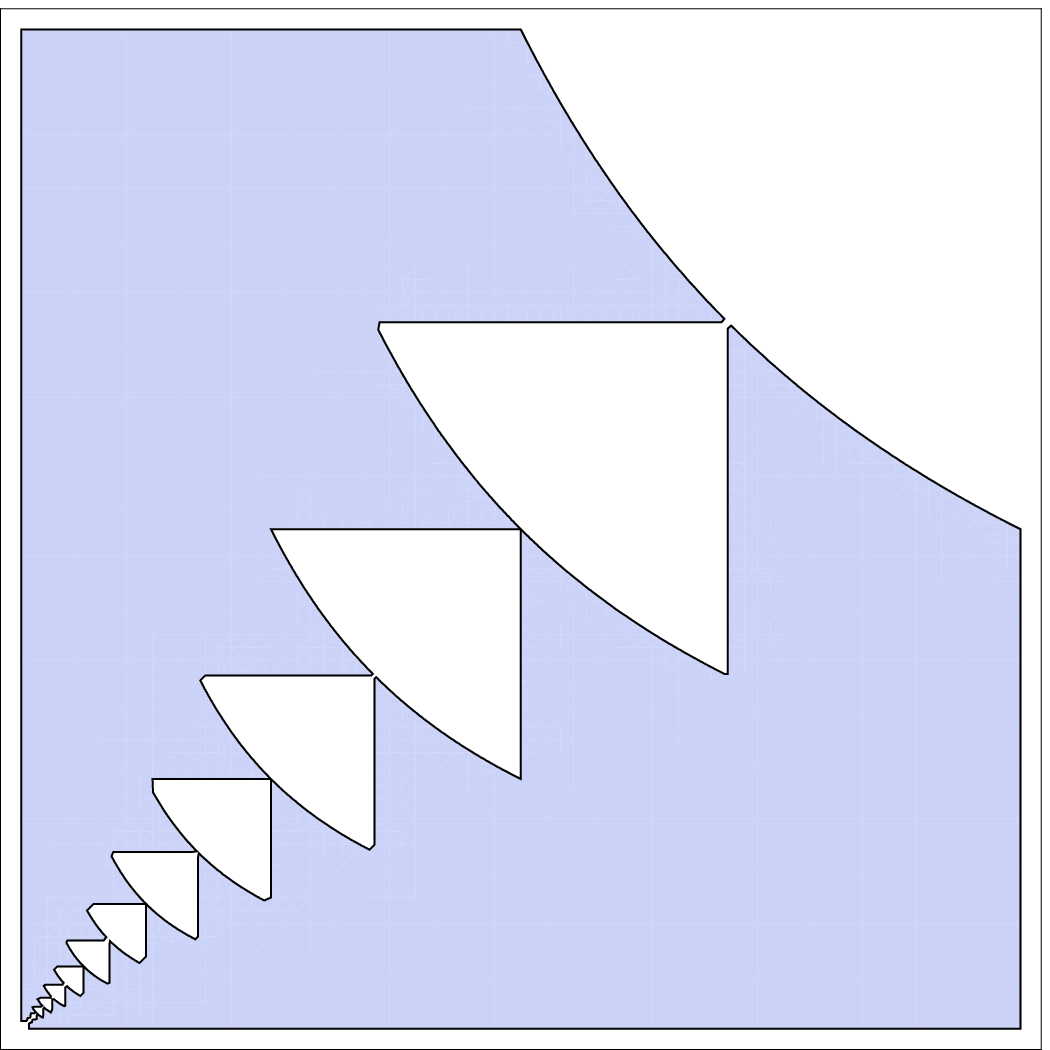, scale=0.5}
\caption{Method of Theorem 5.4}
\end{minipage}
\end{figure}

Theorem 5.4 yields the following bounds for $2 \leq k \leq 6$ (showing an improvement
in the upper bound compared to the data presented in Section 1):

\begin{eqnarray*}
0.555556 \leq & \tilde{c}_2 & \leq 0.727309 \\
0.666667 \leq & \tilde{c}_3 & \leq 0.840690 \\
0.740741 \leq & \tilde{c}_4 & \leq 0.899940 \\
0.796639 \leq & \tilde{c}_5 & \leq 0.935089 \\
0.838889 \leq & \tilde{c}_6 & \leq 0.957139 
\end{eqnarray*}

It seems possible that this technique may also be used to improve the upper bound in the
discrete case. However, the process of iteration creates serious obstacles in the translation 
of Theorem 5.4; every iteration introduces its own error term, and since the number of 
iterations is unbounded, the continuous proof is not sufficient in the discrete setting.

It is worth noting that while the constants $c_k$ and $\tilde{c}_k$ seem similar in nature 
(and indeed we apply similar methods when bounding them), there is no obvious
relation between them; it is not even clear which of these values is larger for a given $k$.

\section{Generalization to $l$-fold-sumfree sets}

A sumfree set $S$ is, by definition, a set such that $f(x,y,z) := x+y-z \neq 0$ for all
$x,y,z \in S$. We can generalize this definition by replacing $f$ with any other
linear form $f(x_1,\ldots,x_n)$ and considering sets such that this form is
nonzero for any $x_1,\ldots,x_n \in S$. 

As a natural generalization, we call $S$ an \emph{$l$-fold-sumfree set} if 
$$\forall x_1,\ldots,x_l, z \in S, \,
x_1 + x_2 + \cdots + x_l - z \neq 0,$$
or equivalently, using sumset notation, if
$$lA \cap A = \emptyset.$$
  We define
$$c_{k,l} := \limsup_{n \to \infty} \frac{1}{n^k} \max \{\#S:S \in \{1,\ldots,n\}^k 
\mbox{ is }l\mbox{-fold sumfree}\},$$
or in the continuous setting,
$$\tilde{c}_{k,l} :=  \max \{\mbox{Vol}(S):S \in [0,1]^k \mbox{ is measurable and }l\mbox{-fold-sumfree}\}.$$

\begin{remark}
In some of the literature (\cite{MR1441239}, for instance), these sets are simply referred to as $l$-sumfree. However, this description is used with various meanings (see \cite{MR2142128}), so we will use the term $l$-fold-sumfree for added clarity.
\end{remark}

As in the sumfree ($l = 2$) case, we can construct large sumfree sets using ``diagonal stripes'', leading to a similar lower bound.

\begin{theorem}
Defining $c_{k,l}$ and $\tilde{c}_{k,l}$ as above,

$$c_{k,l} \geq 1 - \frac{2}{k!}\sum^{\lfloor k/(l+1) \rfloor}_{i=0} (-1)^i \binom{k}{i} \left(\frac{k}{l+1}-i\right)^k,$$
$$\tilde{c}_{k,l} \geq 1 - \frac{2}{k!}\sum^{\lfloor k/(l+1) \rfloor}_{i=0} (-1)^i \binom{k}{i} \left(\frac{k}{l+1}-i\right)^k.$$
\end{theorem}

\begin{proof}
We follow the proof of Theorem 1.1, except now we use the $l$-fold-sumfree set
$$\widetilde{S}(n,k,a) := \left\{(x_1,\ldots,x_k) \in [0,n]^k:\frac{1}{l+1} \leq x_1 + \cdots + x_k < \frac{l}{l+1}\right\}.$$
The given bound is the volume of this set by Lemma 2.1.
\end{proof}

\begin{corollary}
Defining $c_{k,l}$ and $\tilde{c}_{k,l}$ as above,
$$\lim_{k \to \infty} c_{k,l} = \lim_{k \to \infty} \tilde{c}_{k,l} = 1.$$
\end{corollary}

\begin{proof}
The lower bound for $c_{k,l}$ and $\tilde{c}_{k,l}$ given in Theorem 6.1 is larger than the lower bound for $c_k$ given in Theorem 1.1 (as it is the volume of a larger region). Since the previous bound approaches 1 as $k$ grows large, this one does as well.
\end{proof}

Our upper bound does not extend as easily. Adapting our methods, we can deal with the $l = 3$ case, and in fact find an upper bound which is both explicit and reasonably effective; however, it is not evident how to deal with any of the cases where $l \geq 4$.

\begin{theorem}
Defining $c_{k,l}$ and $\tilde{c}_{k,l}$ as above,
$$c_{k,3} \leq 1 - \frac{1}{(1+2^{1/k})^k},$$
$$\tilde{c}_{k,3} \leq 1 - \frac{1}{(1+2^{1/k})^k}.$$
\end{theorem}

\begin{proof}

Let $S$ be an $l$-fold-sumfree subset of $[0,1]^k$ with $\mathrm{Vol}(S) = \alpha$.
We define the sets $$A_1 = S \cap [0,\gamma]^k,$$ $$A_2 = S \cap [1-\gamma,1]^k,$$ 
where $(1-\alpha)^{1/k}< \gamma < \frac{1}{2}$.

Consider the element
$$a:= (1-\gamma,1-\gamma,\ldots,1-\gamma)$$
and suppose that the sets $A_2$ and the translation $A_1 + a \subset [1-\gamma,1]^k$ 
are disjoint. Then
\begin{eqnarray*}
\alpha & \leq & (1 - 2\gamma) + \mathrm{Vol}(A_1) + \mathrm{Vol}(A_2) \\
& = & (1 - 2\gamma) + \mathrm{Vol}(A_1 + a) + \mathrm{Vol}(A_2) \\
& \leq & (1 - 2\gamma) + \gamma = 1 - \gamma,
\end{eqnarray*}
which contradicts our assumption on $\gamma$. 

Thus, there exist elements $w,z \in S$ such that $w+a = z$. This means $S$ cannot contain any
pair of elements $x,y \in S$ such that $x+y=a$, or else we would have $w+x+y=z$, a contradiction
since $S$ is $l$-fold-sumfree. Then, as in the proof of Theorem 5.3, we must have
$$\mathrm{Vol}(S \cap [1-\gamma]^k) \leq \frac{1}{2}(1-\gamma)^k.$$
Using this result and the restriction on $\gamma$,
$$1 - \gamma^k < \alpha \leq 1 - \frac{1}{2}(1-\gamma)^k,$$
and thus
\begin{eqnarray*}
\frac{1}{2}(1-\gamma)^k & < & \gamma^k \\
(1-\gamma) & < & \gamma \cdot 2^{1/k} \\
\frac{1}{2^{1/k}} & < & \gamma.
\end{eqnarray*}

Finally, we use the bound on $\gamma$ to bound $\alpha$ (and thus $\tilde{c}_{k,3}$):
$$\alpha \leq 1- \frac{1}{2}(1-\gamma)^k < 1 - \frac{1}{(1+2^{1/k})^k}.$$

We achieve the upper bound for $c_{k,3}$ using the same sort of integral approximation
technique we applied to Theorems 1.1 and 1.3. The process is virtually identical, so we
omit the details here.
\end{proof}

Theorems 6.1 and 6.3 give us lower and upper bounds for $\tilde{c}_{k,3}$; looking at the cases
where $2 \leq k \leq 6$, we get the following results:

\begin{eqnarray*}
0.750000 \leq & \tilde{c}_{2,3} & \leq 0.828427 \\
0.859375 \leq & \tilde{c}_{3,3} & \leq 0.913360 \\
0.916667 \leq & \tilde{c}_{4,3} & \leq 0.956464 \\
0.949219 \leq & \tilde{c}_{5,3} & \leq 0.978167 \\
0.968620 \leq & \tilde{c}_{6,3} & \leq 0.989061 
\end{eqnarray*}

These bounds illustrate that for 3-fold-sumfree sets, the largest ``diagonal stripe'' sets
have very close to maximal size.

\section{Concluding remarks}

All of the large sumfree (and $l$-fold-sumfree) sets we have constructed are unions of
sets of the form $K_a$ as defined in Section 3. These are certainly the simplest sets to
grasp, but there is no guarantee that the largest sumfree sets have this structure. 

If we limit ourselves to these $K_a$-unions, the problem is simplified to choosing an
optimal sumfree set $A \subset \{k,\ldots,kn\}$. (To conserve space in this section, we 
will use the discrete notation to discuss both the continuous and discrete problems.) 
Since the sets $K_a$ are not of equal size, this is a different task than finding a 
large sumfree set $A$. This suggests a more general combinatorial problem.

\begin{question}
Given an additive set with a weight assigned to each element, what methods can we use
to construct sumfree (resp. $l$-fold-sumfree) sets that maximize the sum of the weights of the elements?
\end{question}

On the other hand, if we relax this structural constraint, we know virtually nothing about whether the upper bound on the size increases.

\begin{question}
Are there optimally large sumfree (resp. $l$-fold-sumfree) subsets of $\{1,\ldots,n\}^k$ which are not 
the union of  ``cross-section'' sets?
\end{question}

Addressing both of these questions would solve the problems we have been studying. Question 7.1 is unlikely to have an answer in full generality, although if the weight distribution is highly structured, as it in this context, there may be methods of approach.

\section{Acknowledgments}
I would like to thank Joseph H. Silverman and Steven J. Miller for their many useful suggestions about the content and presentation of this paper.

\bibliographystyle{plain}
\bibliography{mysources}

\textsc{Department of Mathematics, Brown University, Box 1917, Providence, RI 02912-1917}

\emph{E-mail address}: \texttt{thedan@math.brown.edu}

\end{document}